\documentclass{amsart}

\usepackage{graphicx}

\newtheorem{theorem}{Theorem}[section]
\newtheorem{lemma}[theorem]{Lemma}

\theoremstyle{definition}

\theoremstyle{remark}
\newtheorem{remark}[theorem]{Remark}

\numberwithin{equation}{section}

\usepackage{chngcntr}
\usepackage{graphicx} 
\usepackage{float}

\usepackage{tikz}
\usetikzlibrary{arrows,%
                plotmarks}
\counterwithout{equation}{section}
\counterwithout{theorem}{section}

\begin{document}

\title[Hyperbolic lattices with Zariski-dense surface subgroups]{A hyperbolic lattice in each dimension with Zariski-dense surface subgroups}

\author{Sami Douba}
\address{McGill University, Department of Mathematics and Statistics}
\email{sami.douba@mail.mcgill.ca}
\thanks{The author was supported by a public grant as part of the Investissement d'avenir project, FMJH, by LabEx CARMIN, ANR-10-LABX-59-01, and by the National Science Centre, Poland UMO-2018/30/M/ST1/00668.}


\begin{abstract}
For each integer $n \geq 3$, we exhibit a nonuniform arithmetic lattice in $\mathrm{SO}(n,1)$ containing Zariski-dense surface subgroups.
\end{abstract}

\maketitle

It follows from a straightforward ping pong argument that any lattice in $\mathrm{SO}(n,1)$, $n \geq 2$, contains a Zariski-dense copy of a noncocompact discrete subgroup of $\mathrm{SO}(2,1)$, namely, a Zariski-dense free subgroup (this may also be seen as a consequence of the Borel density theorem \cite[Corollary~4.3]{MR123639} and \cite[Theorem~3]{MR286898}). It is thus natural to look for Zariski-dense subgroups of lattices in $\mathrm{SO}(n,1)$ that are isomorphic to cocompact discrete subgroups of, that is, uniform lattices in, $\mathrm{SO}(2,1)$; the latter are virtually fundamental groups of closed hyperbolic surfaces. We refer to fundamental groups of such surfaces as {\it surface groups}. 

While nonuniform lattices in $\mathrm{SO}(2,1)$ are virtually free and hence do not possess surface subgroups, it is expected that any lattice in $\mathrm{SO}(n,1)$ for $n \geq 3$ contains surface subgroups, and even Zariski-dense such subgroups. This has been established for $n=3$ by Cooper, Long, and Reid \cite{MR1431827} in the nonuniform case (see also \cite{MR3921320, MR4255043}) and by Kahn and Markovic \cite{MR2912704} in the uniform case; and for odd $n \geq 3$ by Hamenst\"adt \cite{MR3361773} (see also \cite{kahn2018surface}) in the uniform case. However, while some standard constructions of closed arithmetic hyperbolic manifolds of arbitrary dimension contain immersed totally geodesic surfaces (see, for instance, \cite[Example~8]{benoist2004five}), the author was not aware of an example in the literature of a lattice in $\mathrm{SO}(n,1)$ for each $n \geq 3$ containing {\it Zariski-dense} surface subgroups. The purpose of this note is to present such an example in each dimension.

\begin{theorem}\label{main}
For each $n \geq 3$, there is a nonuniform arithmetic lattice in $\mathrm{SO}(n,1)$ containing Zariski-dense surface subgroups.
\end{theorem}

Before we proceed, we fix some notation. Given an integer $n \geq 1$, a subdomain~$D \subset \mathbb{R}$, and a symmetric matrix $Q \in \mathrm{M}_{n+1}(D)$ of signature $(n,1)$, we denote by $\mathrm{O}(Q; D)$ (resp., $\mathrm{SO}(Q; D)$) the set of all matrices~$g \in \mathrm{GL}_{n+1}(D)$ (resp., $g \in \mathrm{SL}_{n+1}(D)$) satisfying $g^T Q g = Q$. The hypersurface of $\mathbb{R}^{n+1}$ consisting of all~$x \in \mathbb{R}^{n+1}$ satisfying $x^T Q x = -1$ is a two-sheeted hyperboloid; we denote by $\mathrm{O}'(Q; D)$ the subgroup of $\mathrm{O}(Q; D)$ preserving each sheet. When $Q$ is the standard form $\mathrm{diag}(1, \ldots, 1, -1)$, we write $\mathrm{O}(n,1)$ (resp., $\mathrm{SO}(n,1)$, $\mathrm{O}'(n,1)$) in the place of $\mathrm{O}(Q; \mathbb{R})$ (resp., $\mathrm{SO}(Q; \mathbb{R})$, $\mathrm{O}'(Q; \mathbb{R})$). We view $\mathrm{O}'(n,1)$ as the isometry group of~$n$-dimensional real hyperbolic space $\mathbb{H}^n$ via the hyperboloid model of the latter.

We now proceed to the examples. For $m \geq 3$, let $K_m \in \mathrm{M}_m(\mathbb{Z})$ be the matrix all of whose entries are equal to $1$; let $B_m \in \mathrm{M}_m(\mathbb{Z})$ be the matrix with $2$'s on the diagonal, $1$'s on the superdiagonal and subdiagonal, and $0$'s everywhere else; and let $Q_m = B_m - K_m$. What follows is the key observation of this note.

\begin{lemma}\label{keyobs}
The symmetric matrix $Q_{n+1}$ has signature $(n,1)$ for $n \geq 3$.
\end{lemma}

\begin{proof}
The author is grateful to Yves Benoist for the following efficient argument. For $m > 0$, let $v_m := (1, \ldots, 1)^T \in \mathbb{R}^m$ and let $H_{m-1}$ be the orthogonal complement of $\langle v_m \rangle$ in $\mathbb{R}^m$ with respect to the standard inner product on $\mathbb{R}^m$. Now let ${n \geq 3}$. Then $v_{n+1}^T Q_{n+1} v_{n+1} = -n^2+2n+1 < 0$, so it suffices to show that the restriction of~$Q_{n+1}$ to~$H_n$ is positive-definite. Indeed, since the restriction of the form $K_{n+1}$ to~$H_n$ is trivial, we have that the forms $Q_{n+1}$ and $B_{n+1}$ have the same restriction to~$H_n$, so that it is enough to show that $B_{n+1}$ is positive-definite. This is true since we may view $B_{n+1}$ as the matrix representation  of the standard inner product on~$\mathbb{R}^{n+2}$ restricted to~$H_{n+1}$ with respect to the basis $((-1)^{i+1}(e_i -e_{i+1}))_{i=1}^{n+1}$ of~$H_{n+1}$, where $(e_1, \ldots, e_{n+2})$ is the standard basis for $\mathbb{R}^{n+2}$. 
\end{proof}

\begin{proof}[Proof of Theorem \ref{main}]
Let $n \geq 3$. By the Borel–Harish-Chandra theorem \cite{MR147566} (and by Lemma \ref{keyobs}), we have that $\Lambda_n := \mathrm{O}'(Q_{n+1}; \mathbb{Z})$ is a lattice in  $\mathrm{O}(Q_{n+1}; \mathbb{R})$. As will become clear in the course of the proof, the lattice $\Lambda_n$ is nonuniform. We show that $\Lambda_n$ contains a Zariski-dense subgroup isomorphic to a cocompact lattice in $\mathrm{O}'(2,1)$ generated by the reflections in the sides of a hyperbolic right-angled~$2n$-gon. This will complete the proof since $Q_{n+1}$ has signature $(n,1)$ by Lemma \ref{keyobs}.

To that end, let $(W_n, (s_1, \ldots, s_{n+1}))$ be the right-angled Coxeter system associated to the matrix $Q_{n+1}$; that is, let $W_n$ be the right-angled Coxeter group given by the presentation with generators $s_1, \ldots, s_{n+1}$ subject to the relations $s_i^2 = 1$ for~$i = 1, \ldots, {n+1}$ and $[s_i, s_j] = 1$ whenever the $(i,j)^\text{th}$ entry of $Q_{n+1}$ is $0$ (in our case, the latter happens exactly when $|i-j|=1$). The image $\Gamma_n$ of the Tits representation $\rho_n: W_n \rightarrow \mathrm{SL}^\pm_{n+1}(\mathbb{R})$ associated to the Coxeter system $(W_n, (s_1, \ldots, s_{n+1}))$ lies in~$\Lambda_n$ and is Zariski-dense in $\mathrm{O}(Q_{n+1}; \mathbb{R})$ \cite{benoist2004adherence}. Interpreting $\mathrm{O}'(Q_{n+1}; \mathbb{R})$ as the group of conformal diffeomorphisms of $\mathbb{S}^{n-1} = \partial \mathbb{H}^n$ by conjugating $\mathrm{O}(Q_{n+1}; \mathbb{R})$ to $\mathrm{O}(n,1)$ within $\mathrm{GL}_{n+1}(\mathbb{R})$, we have that $\gamma_i := \rho_n(s_i)$ is an inversion in a (round) hypersphere~$S_i$ of $\mathbb{S}^{n-1}$ for $i= 1, \ldots, n+1$. Moreover, we have that $S_i$ is orthogonal to $S_{i+1}$ for~$i=1, \ldots, n$, and that $S_i$ and $S_j$ are tangent for $1 \leq i < j-1 \leq n$. The latter follows from the fact that $\gamma_i \gamma_j$ is nontrivial and unipotent, hence parabolic, for such~$i$ and $j$.

We now visualize $\mathbb{S}^{n-1}$ via stereographic projection onto $\mathbb{R}^{n-1}$ from the tangency point $\infty$ of $S_1$ and $S_{n+1}$. Under this projection, the hyperspheres $S_1$ and $S_{n+1}$ are parallel hyperplanes of $\mathbb{R}^{n-1}$, while the remaining hyperspheres are contained in some ball $B \subset \mathbb{R}^{n-1}$. Since the stabilizer $\mathrm{Stab}_{\Lambda_n}(\infty)$ of $\infty$ in $\Lambda_n$ contains the reflections $\gamma_1$ and $\gamma_{n+1}$ in the parallel Euclidean hyperplanes $S_1$ and $S_{n+1}$, respectively, and since $\Lambda_n$ is a lattice in $\mathrm{O}'(Q_{n+1}; \mathbb{R})$, we must have that $\mathrm{Stab}_{\Lambda_n}(\infty)$ is a lattice in $\mathrm{Isom}(\mathbb{R}^{n-1})$ by the Margulis lemma (see, for instance, \cite[Prop.~D.2.6]{MR1219310}). In particular, there is some translation $\sigma \in \mathrm{Stab}_{\Lambda_n}(\infty)$ with nontrivial $S_1$-component, and so $\tau := (\gamma_1 \sigma \gamma_1) \sigma$ is a nontrivial translation parallel to $S_1$ (and $S_{n+1}$). We now replace $\tau$ with a sufficiently high power so that $B \cap \tau(B) = \emptyset$. Appropriately defined, the common exterior of the hyperspheres $S_1, \ldots, S_{n+1}$ and their images under $\tau$ produce a Coxeter polytope in $\mathbb{H}^n$ with the correct dihedral angles so that $\langle \Gamma_n, \tau \Gamma_n \tau^{-1} \rangle < \Lambda_n$ is isomorphic to the right-angled $2n$-gon group (see, for example, the introduction of \cite{MR783604}).
\end{proof}

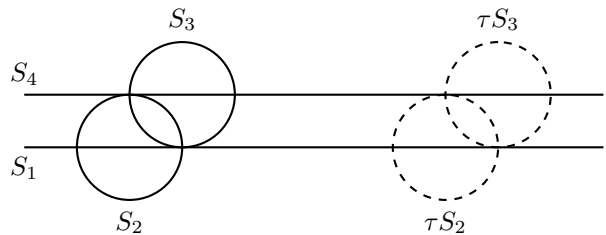
\begin{figure}
   \begin{tikzpicture}[scale=0.7, rotate=90]
    \draw[thick,black] (0,-7)--(0,4) node[below] {$S_1$};
    \draw[thick,black] (1,-7)--(1,4) node[above] {$S_4$};
    \draw[black,thick] (0,2) circle (1cm) node[yshift=-1cm] {$S_2$};
    \draw[black,thick] (1,1) circle (1cm) node[yshift=1cm] {$S_3$};
    \draw[black,thick, dashed] (0,-4) circle (1cm) node[yshift=-1cm] {$\tau S_2$};
    \draw[black,thick, dashed] (1,-5) circle (1cm) node[yshift=1cm] {$\tau S_3$};
  \end{tikzpicture}
  \caption{Visualizing the case $n=3$. The sphere $\mathbb{S}^2$ is stereographically projected onto the plane from the tangency point of the circles $S_1$ and $S_4$, so that $S_1$ and $S_4$ project to parallel lines. Up to a Euclidean similarity, the circles $S_i$ are as in the figure. Zariski-density in $\mathrm{O}(Q_4; \mathbb{R})$ of the subgroup $\Gamma_3$ can be deduced from the fact that no vertical line is orthogonal to both $S_2$ and $S_3$. Any lattice in $\mathrm{O}'(3,1)$ containing the inversions in the $S_i$ also contains inversions in two circles resembling the dashed circles above; the latter represent the images of~$S_2$ and $S_3$ under a horizontal translation $\tau$ of large magnitude. The subgroup of $\mathrm{O}'(3,1)$ generated by the inversions in the above six circles is (abstractly) a right-angled hexagon group. Extending each circle $S_i$ to a sphere~$S_i'$ in~$\mathbb{S}^3$ orthogonal to the page, and denoting by $S_5'$ the Euclidean plane parallel to the page and resting on top of $S_2'$ and~$S_3'$, we have that the subgroup of~$\mathrm{O}'(4,1)$ generated by the inversions in~$S_1', \ldots, S_5'$ is (conjugate to) the right-angled pentagon group $\Gamma_4'$ in Remark~\ref{evendim}. }
  \label{fig:circles}
\end{figure}

\begin{remark}
A ping pong argument following \cite[Section~VII.E]{MR959135} demonstrates that in fact $\langle \Gamma_n, \tau \rangle < \Lambda_n$ decomposes as the HNN extension $\Gamma_n *_{\langle \gamma_1, \gamma_{n+1}\rangle}$ given by the identity map on ${\langle \gamma_1, \gamma_{n+1}\rangle}$. 
\end{remark}

\begin{remark}
The surface subgroups produced above are geometrically finite but are not convex cocompact since they contain parabolics. They are also automatically {\it thin} in the sense of Sarnak \cite{MR3220897} since a surface group cannot be realized as a lattice in $\mathrm{O}'(n,1)$ for $n \geq 3$ (for instance, because the outer automorphism group of such a lattice is finite by Mostow–Prasad rigidity \cite{mostow1968quasi, prasad1973strong}). 
\end{remark}

\begin{remark}\label{zariski}
In this remark, we use the language of Coxeter schemes following \cite[Section~II.5]{MR783604}. To justify Zariski-density of $\Gamma_n$, and hence $\langle \Gamma_n, \tau \Gamma_n \tau^{-1} \rangle$, in~$\mathrm{O}(Q_{n+1}; \mathbb{R})$, we appealed to the general result of Benoist and de la Harpe \cite{benoist2004adherence}, which asserts in particular that if the Gram matrix~$Q$ of a finite connected Coxeter scheme~$\Sigma$ with no dotted edges is nondegenerate, then the Tits representation of the associated Coxeter group is Zariski-dense in $\mathrm{O}(Q; \mathbb{R})$. When~$Q$ has a single negative eigenvalue (as is the case for $Q = Q_{n+1}$,  $n \geq 3$, by Lemma~\ref{keyobs}), so that $Q$ is the Gram matrix of a hyperbolic Coxeter polytope \cite[Theorem~2.1]{MR783604}, this also follows from the fact that, for $n \geq 2$, if $P \subset \mathbb{H}^n$ is an irreducible Coxeter polytope with finitely many bounding hyperplanes $\Pi_i$, then the subgroup of~$\mathrm{O}'(n,1)$ generated by the reflections in the $\Pi_i$ is Zariski-dense in $\mathrm{O}(n,1)$ if and only if the~$\Pi_i$ do not all share a point in $\mathbb{H}^n \cup \partial \mathbb{H}^n$ or a common orthogonal hyperplane in $\mathbb{H}^n$ (see, for instance, \cite[Theorem~1.3]{MR1836778}), which holds if and only if the Gram matrix of $P$ has rank $n+1$ \cite[Section~I.1]{MR783604}. 

We remark that if all the edges of the Coxeter scheme $\Sigma$ are bold (in other words, if the entries of $Q$ are contained in $\{-1,0,1\}$, as is true for $Q = Q_{n+1}$), then the argument of Benoist and de la Harpe simplifies. For an outline of their argument in this case, see the proof of Lemma 2 in \cite{douba2022thin}.
\end{remark}

\begin{remark}\label{moreefficient}
Let $n \geq 4$. We have demonstrated that $\Lambda_n$ contains a Zariski-dense copy of the right-angled $2n$-gon group, but it is even true that $\Lambda_n$ contains a Zariski-dense copy of the right-angled $2(n-1)$-gon group. Indeed, by Lemma \ref{keyobs} and Remark~\ref{zariski}, there is a unique hypersphere $S \subset \mathbb{S}^{n-1}$ that is simultaneously orthogonal to $S_1, \ldots, S_n$. We visualize $\mathbb{S}^{n-1}$ via stereographic projection onto $\mathbb{R}^{n-1}$ from the tangency point of $S_1$ and $S_n$. Under this projection, the hyperspheres~$S$,~$S_1$, and~$S_n$ are hyperplanes of $\mathbb{R}^{n-1}$, while $S_2, \ldots, S_{n-1}$  are Euclidean $(n-2)$-spheres. As in the proof of Theorem \ref{main}, there is some Euclidean translation in $\Lambda_n$ that is not parallel to $S$, and hence some Euclidean translation in $\Lambda_n$ that is parallel to~$S_1$ but not parallel to $S$. For a sufficiently high power $\tau'$ of the latter translation, we have that~$\langle \gamma_1, \gamma_2, \ldots, \gamma_n, \tau' \gamma_{n-1} \tau'^{-1}, \tau' \gamma_{n-2} \tau'^{-1}, \ldots, \tau' \gamma_2\tau'^{-1} \rangle < \Lambda_n$ is a right-angled~$2(n-1)$-gon group. Moreover, by Remark \ref{zariski}, this subgroup of $\Lambda_n$ is Zariski-dense in~$\mathrm{O}(n,1)$ since there is no hypersphere in $\mathbb{S}^{n-1}$ that is simultaneously orthogonal to $S_1, S_2, \ldots, S_n, \tau' S_{n-1}, \tau' S_{n-2}, \ldots, \tau'S_2$. 
\end{remark}

\begin{remark}\label{evendim} There are more efficient examples in even dimensions. Indeed, let $n \geq 4$, and $Q_{n+1}' \in \mathrm{M}_{n+1}(\mathbb{Z})$ be the matrix obtained from $Q_{n+1}$ by replacing the top-right and bottom-left entries with $0$'s.  Let $(W_n', (t_1^{(n)}, \ldots, t_{n+1}^{(n)}))$ be the right-angled Coxeter system associated to $Q_{n+1}'$, so that~$W_n'$ is a right-angled $(n+1)$-gon group. The associated Tits representation $\rho_n' : W_n' \rightarrow \mathrm{SL}^\pm_{n+1}(\mathbb{R})$ realizes~$W_n'$ as a subgroup~$\Gamma_n'$ of~$\mathrm{O}(Q_{n+1}'; \mathbb{Z})$ in $\mathrm{O}(Q_{n+1}'; \mathbb{R})$. If $n$ is even, then $Q_{n+1}'$ has signature $(n,1)$  \cite[Example~4]{douba2022thin}, and so again by \cite{benoist2004adherence} (or Remark \ref{zariski}), we have that~$\Gamma_n'$ is Zariski-dense in $\mathrm{O}(Q_{n+1}'; \mathbb{R})$. In this manner (alternatively, via Remark~\ref{moreefficient}), one for instance obtains a nonuniform arithmetic lattice in $\mathrm{SO}(4,1)$ containing a Zariski-dense copy of the fundamental group of a closed orientable genus-$2$ surface. 

Now suppose instead that $n$ is odd. Then $Q_{n+1}'$ has signature $(n-1,1,1)$, with kernel spanned by the vector $u_{n+1} := ((-1)^i)_{i=0}^n \in \mathbb{R}^{n+1}$. Note that $\Gamma_n'$ is contained in the stabilizer $G_n$ of $u_{n+1}$ in $\mathrm{O}(Q_{n+1}'; \mathbb{R})$. Denoting by $V_n$ the quotient of $\mathbb{R}^{n+1}$ by the span of $u_{n+1}$, by $\overline{Q_{n+1}'}$ the form induced on~$V_n$ by $Q_{n+1}'$, and by~$\mathrm{O}(\overline{Q_{n+1}'})$ the group of linear automorphisms of $V_n$ preserving the form $\overline{Q_{n+1}'}$, we have a natural map $G_n \rightarrow \mathrm{O}(\overline{Q_{n+1}'})$. Since $Q_n$ is the matrix representation of the form~$\overline{Q_{n+1}'}$ with respect to the basis $(\overline{e_1}, \ldots, \overline{e_n})$ for $V_n$, where $\overline{e_i}$ is the image in $V_n$ of the $i^\text{th}$ standard basis vector for $\mathbb{R}^{n+1}$, we may identify $\mathrm{O}(\overline{Q_{n+1}'})$ with~$\mathrm{O}(Q_n; \mathbb{R})$ to obtain a map ${\pi_n : G_n \rightarrow \mathrm{O}(Q_n; \mathbb{R})}$; explicitly, this map sends a matrix ${A = (a_{i,j})_{i,j} \in G_n}$ to the matrix obtained from $A$ by first adding $a_{n+1, j}u_{n+1}$ to the $j^\text{th}$ column for~$1 \leq j \leq n$ and then deleting the final row and column. In particular, we have that~${\pi_n(\Gamma_n') \subset \mathrm{O}(Q_n; \mathbb{Z})}$ and that $\pi_n(\rho_n'(t_i^{(n)})) = \rho_{n-1}(s_i)$ for~$i = 1, \ldots, n$. Moreover, the map $\pi_n$ is injective on $\Gamma_n'$; see \cite{de2012semisimple} and the references therein.  The conclusion is that  $\Gamma_m$ is in fact contained in a right-angled~$(m+2)$-gon subgroup of $\mathrm{O}(Q_{m+1}; \mathbb{Z})$, namely, $\pi_{m+1}(\Gamma_{m+1}')$, for $m \geq 4$ even. 
\end{remark}

\subsection*{Acknowledgements} I thank Yves Benoist for the proof of Lemma \ref{keyobs}. I am grateful to Pierre Pansu for his hospitality during my stay at Universit\'e Paris-Saclay, where this note was written, and to Konstantinos Tsouvalas for helpful discussions. 
\bibliography{ThinRACGsbib}{}
\bibliographystyle{alpha}

\end{document}